\documentclass{article}

\usepackage{setspace}
\usepackage{color,soul}
\usepackage{amsfonts}
\usepackage{dsfont}
\usepackage[cmex10]{amsmath}
\usepackage{amssymb}
\usepackage{multirow}

\usepackage[table]{xcolor}

\usepackage{algorithm}%
\usepackage{algpseudocode}%
\algnewcommand{\algorithmicgoto}{\textbf{go to}}%
\algnewcommand{\Goto}[1]{\algorithmicgoto~\ref{#1}}%

\usepackage[mathscr]{euscript}

\usepackage{graphicx}
\graphicspath{ {./} }
\DeclareGraphicsExtensions{.pdf}

\newtheorem{theorem}{Theorem}%
\newtheorem{lemma}{Lemma}%
\newtheorem{proof}{Proof}%

\usepackage{framed} %

\usepackage{geometry} %

\newenvironment{keywords}{
	\vspace{1em}\noindent\textbf{Keywords:}\ } %
{}

\begin{document}

\newgeometry{top=72pt,bottom=54pt,right=54pt,left=54pt}

\title{Optimal Diffusion Processes}
\author{Saber Jafarizadeh \\
	Member~IEEE \\
	Rakuten Institute of Technology, Rakuten Crimson House, Tokyo, Japan \\
	\texttt{saber.jafarizadeh@rakuten.com}}

\date{}
\maketitle
\thispagestyle{empty} %

\bibliographystyle{plain}

\begin{abstract}

Of stochastic differential equations, diffusion processes have been adopted in numerous applications, as more relevant and flexible models.
This paper studies diffusion processes in a different setting, where for a given stationary distribution and average variance, it seeks the diffusion process with optimal convergence rate.
It is shown that
the optimal drift function is a linear function and
the convergence rate of the stochastic process is bounded by the ratio of the average variance to the variance of the stationary distribution.
Furthermore, the concavity of the optimal relaxation time as a function of the stationary distribution has been proven, and it is shown that all Pearson diffusion processes of the Hypergeometric type with polynomial functions of at most degree two as the variance functions are optimal.

\end{abstract}


\section{Introduction}

Due to random nature of any real signal, propagation media, interference and even information itself,
the
stochastic %
approach has been one of the popular choices
towards problems associated with information transmission, as well as other applications in physics, mathematical biology, queueing, electrical engineering, economic and finance
\cite{Pereyra2016,Albanese2001D02,Linetsky2006,Mendoza2009,William2015,EXARCHOS2018159,Farsad2016}.
A particularly effective application of the stochastic approach is the modeling of the phenomena as random processes which are the solution of an appropriate Stochastic Differential Equation (SDE),
i.e., an ordinary differential equation with random excitation, due to the presence of noise in the dynamics of the process \cite{Iacus2008,Bishwal2008}.
The basic idea behind this approach is that the state of the studied process is the solution of an SDE with a priori information, such as
the stationary %
distribution.

SDEs are generally identified by two main parameters, the drift and the variance functions.
Based on dependency of the drift and variance functions on the previous values of the process as well as the present values, the SDEs are categorized into It\^{o} and diffusion processes.
In diffusion processes (which are the focus of studies in this paper), the drift and variance functions depend only on the present value of the process and their solution satisfy the Markov property \cite{Borodin2017},
which results in a unique stationary distribution for the diffusion process.
Diffusion processes with invariant equilibrium distributions have found applications in approximating the process of sampling from high dimensional probability distributions 
\cite{Wu2014}.
One of the important factors in such applications is the convergence rate of the diffusion process and its optimizations. 
\cite{Hwang1993}
is regarded as one of the pioneering works that has studied the optimization of convergence rate of the diffusion processes.
In
\cite{Hwang1993,Wu2014}
(and the following works of
\cite{Hwang1993}),
optimization of the convergence rate of the non-reversible diffusion processes have been addressed, but at the expense of limiting the variance function to be a constant function. 
This is unlike the studied presented in this paper, where no such restrictions have been applied on the variance function. 
In a nutshell, in this paper, for different types of diffusion processes with given average variance function and stationary distribution,
we have provided the diffusion process with the optimal convergence rate,
which has been possible by adjust the variance function.
In detail, for a given distribution $\pi(x)$,
our goal is
to construct the diffusion process with optimal convergence rate whose stationary distribution is $\pi(x)$.

We
have shown that
the optimal drift function is a linear function and
the optimal convergence rate of the diffusion process is equal to the ratio of the average variance to the variance of the stationary distribution.
Therefore, for all stochastic processes with given stationary distribution $\pi(x)$ and average variance of $\widehat{\sigma}^{2}/2$, the convergence rate of the stochastic process is bounded by $\frac{ \widehat{\sigma}^{2}/2 }{ m_2 - m_1^2 }$, where $m_2 - m_1^2$ is the variance of the stationary distribution $\pi(x)$.
Furthermore, we have proved the concavity of the optimal relaxation time $(\tau)$ as a function of the stationary distribution $(\pi(x))$.

We have shown that all Pearson diffusion processes of the Hypergeometric type with polynomial functions of at most degree two as the variance functions are optimal.
In a table, we have provided the optimal convergence rate along with other attributes for the Pearson diffusion processes. Pearson diffusion processes are diffusion processes
whose %
drift and variance functions are polynomials of degrees at most one and two, respectively.
Also, we have provided the optimal results for two other types of diffusion processes, namely
(i) diffusion processes with polynomial functions of degree higher than two as the variance function, and
(ii) diffusion processes with hyperexponential distribution as the stationary distribution, which is an example of a diffusion process with a non-polynomial function as its variance function.

The rest of the paper is organized as follows.
In Section \ref{Sec:Preliminaries}, Preliminaries on Fokker-Plank equations are presented.
Our general approach towards optimization of the diffusion processes has been provided in Section \ref{sec:OptimalDiffusionProcesses}.
In Section \ref{sec:MainResults}, main results of the paper have been presented,
and in Section \ref{sec:examples}, the optimal results for three different types of diffusion processes have been reported.
and Section \ref{sec:Conclusions} concludes the paper.

\section{Preliminaries}
\label{Sec:Preliminaries}

\subsection{Fokker-Plank Equation }
Let's consider the following time-homogeneous diffusion process (also known as It\^{o} diffusion) $X(t)$,
\begin{equation}
    \label{eq:EqDiffusionTimeHomogeneous}
    \begin{gathered}
        \mathrm{d} X(t) = \mu( X(t) ) \, \mathrm{d} t + \sigma ( X(t) ) \, \mathrm{d} W(t),
     \end{gathered}
\end{equation}
where $W(t)$ is a Brownian motion in $\mathbb{R}$, The function $\mu$ is the drift coefficient of $X$ and $\sigma$ is known as the diffusion coefficient of $X$ and it can be interpreted as a variance.
Imposing any control term to either the drift coefficient or the diffusion coefficient will transform the diffusion process (\ref{eq:EqDiffusionTimeHomogeneous}) into a controlled diffusion process \cite{Krylov1995,Krylov1980}.
The domain of $X(t)$ is denoted by $I = ( x_1 , x_2 )$ where $-\infty \leq x_1 < x_2 \leq \infty$.

The process
$X(t)$ %
is Markov, and its transition density
$p(t;x,y)$ %
is defined as
\begin{equation}
    \label{eq:267}
    \begin{gathered}
        P\left( X(t) \in A | X_0 = x \right)  = \int_{A}  p(t;x,y)  dy,
     \end{gathered}
\end{equation}
for $t \geq 0$ and any Borel set $A \subseteq I$.
Under regularity conditions, the transition density solves the backward Kolmogorov equation,
\begin{equation}
    \label{eq:20170726284}
    \begin{gathered}
        \frac{ \partial p(t;x,y) }{ \partial t }  =  \mathcal{A} \left[ \mu , \sigma^2 \right] p(t;x,y),
        \quad t>0,
        \quad (x,y) \in I,
     \end{gathered}
\end{equation}
where $\mathcal{A}$ is an infinitesimal generator defined as below,
\begin{equation}
    \label{eq:20170726296}
    \begin{gathered}
        \mathcal{A} \left[ \mu , \sigma^2 \right] p(t;x,y)
        =
        \mu(x) \frac{ \partial p(t;x,y) }{ \partial x }
        +
        \frac{1}{2} \sigma^{2}(x) \frac{ \partial^2 p(t;x,y) }{ \partial x^2 }.
     \end{gathered}
\end{equation}
$p(t;x,y)$ %
can be completely characterised by the generator $\mathcal{A} \left[ \mu , \sigma^2 \right]$.
Therefore, it is a functional of $\mu$ and $\sigma^2$ and it can be written as
$p(t;x,y)  =  p\left(t;x,y ; \mu, \sigma \right)$. %
Suppose that
$\{X(t)\}$  %
is strictly stationary and ergodic, in which case it has a stationary marginal density which we denote $\pi(x)$,
$P\left( X(t) \in A \right)  =  \int_A \pi(x) dx$, %
for $t \geq 0$ and any Borel set $A \subseteq I$.
Note that $\pi$ is invariant, i.e.
\begin{equation}
    \label{eq:20170727326}
    \begin{gathered}
        \pi(y)  =  \int_{I} p(t;x,y) \pi(x) dx,
        \quad   \text{for}   \quad   t \geq 0.
    \end{gathered}
\end{equation}
Alternatively, the characterization of the transition density can be obtained through the forward Kolmogorov equation,
\begin{equation}
    \label{eq:278}
    \begin{gathered}
        \frac{\partial  p(t;x,y) }{\partial t}
        =
        - \frac{ \partial \left( \mu(y) p(t;x,y) \right) }{ \partial y }
        +
        \frac{1}{2} \frac{ \partial^2 \left( \sigma^{2}(y) p(t;x,y) \right) }{ \partial y^2 },
     \end{gathered}
\end{equation}
for $t > 0$,
$x \in I$ %
with initial value
$p(0;x,y) = \delta(y - x)$ %
and the following boundary conditions, which are corresponding to reflecting barriers,
\begin{equation}
    \label{eq:FokkerPlankBoundaryConditions}
    \begin{gathered}
        \left. \frac{1}{2} \frac{\partial}{ \partial y } \left( \sigma^2(y) p(t;x,y) \right)  -  \mu(y) p(t;x,y) \right|_{ y = y_1, y_2 }  =  0.
     \end{gathered}
\end{equation}
Multiplying (\ref{eq:278}) with $\pi(x)$ and integrating over $x$ and using (\ref{eq:20170727326}) the following is obtained,
\begin{equation}
    \label{eq:281}
    \begin{gathered}
        \frac{d}{dy} \left( -\mu(y) \pi(y) + \frac{1}{2}\frac{d}{dy} \left( \sigma^2(y) \pi(y) \right) \right)  =  0.
     \end{gathered}
\end{equation}
If the detailed balance condition is satisfied, i.e.
\begin{equation}
    \label{eq:292}
    \begin{gathered}
        -\mu(y) \pi(y) + \frac{1}{2} \frac{d}{dy} \left( \sigma^2(y) \pi(y) \right) = 0,
     \end{gathered}
\end{equation}
then the process $X(t)$ is reversible wrt $\pi(x)$.
Using the detailed balance condition (\ref{eq:292}), we can conclude
\begin{equation}
    \label{eq:20170727442}
    \begin{gathered}
        \mu(x) = \frac{1}{2 \pi(x)} \frac{\partial}{\partial x} \left( \sigma^2(x) \pi(x) \right).
     \end{gathered}
\end{equation}
All the control effort will be applied to the variance term, since considering (\ref{eq:20170727442}), the drift term can be derived in terms of the variance term.
(\ref{eq:20170727442}) is known as Pearson equation
\cite{Pearson1914} %
if $\mu(x)$ and $\sigma^2(x)/2$ are a polynomial of degrees at most one and two, respectively.
Using (\ref{eq:20170727442}), the backward (\ref{eq:20170726284}) and forward Kolmogorov (\ref{eq:278}) equations reduce to the following forms,
\begin{equation}
    \label{eq:20170727453}
    \begin{gathered}
        \frac{\partial}{\partial t} p(t;x,y)
        =
        \frac{1}{2 \pi(x)} \frac{\partial}{\partial x} \left( \sigma^2(x) \pi(x) \frac{\partial}{\partial x}  p(t;x,y) \right)
        \quad \text{for} \quad t>0,
     \end{gathered}
\end{equation}
\begin{equation}
    \label{eq:20170727461}
    \begin{gathered}
        \frac{\partial}{\partial t} \frac{p(t;x,y)}{\pi(y)}
        =
        \frac{1}{2\pi(y)} \frac{\partial}{\partial y} \left( \sigma^2(y) \pi(y) \frac{\partial}{\partial y} \left( \frac{ p(t;x,y) }{ \pi(y) } \right) \right)
        \quad \text{for} \quad t>0.
     \end{gathered}
\end{equation}
By separating the spatial and time variables in (\ref{eq:20170727453}) and (\ref{eq:20170727461}), we have
\begin{equation}
    \label{eq:20180217489}
    \begin{gathered}
        \frac{1}{2} \frac{d}{dx} \left( \pi(x) \sigma^2(x) \frac{d}{dx} \phi(x) \right)  +  \lambda \pi(x) \phi(x)  =  0,
     \end{gathered}
\end{equation}
with boundary conditions
\begin{equation}
    \label{eq:20180217500}
    \begin{gathered}
        \left. \frac{ \sigma^2(x) }{ 2 } \pi(x) \frac{d}{dx}\phi(x) \right|_{x_1,x_2}  =  0.
     \end{gathered}
\end{equation}
From the classical Sturm-Liouville theory, it is known that the spectrum of
(\ref{eq:20180217489}) %
is discrete if $x_1$ and $x_2$ are finite.
If one or both of the boundaries are infinite, then a continuous range of eigenvalues may be present.
In terms of the solution of
(\ref{eq:20180217489}) %
the principal solution
$p(t;x,y)$ %
of
(\ref{eq:278}) %
can be written as below,
\begin{equation}
    \label{eq:Eq372} %
    \begin{gathered}
        p(t;x,y)
        =
        \pi(y)
        (  \sum_{n} e^{-t \lambda_n} \phi_n(y) \phi_n(x)
                +
                \int e^{-t \lambda} \phi(\lambda, x) \phi(\lambda, y) d\lambda
        ),
     \end{gathered}
\end{equation}
where the summation is taken over all discrete eigenvalues and the integral is taken over the continuous range of eigenvalues.
The eigenfunctions $\phi(x)$ in (\ref{eq:Eq372}) are assumed to be normalized so that corresponding to discrete eigenvalues we have
\begin{equation}
    \label{eq:Eq386}
    \begin{gathered}
        \int_{x_1}^{x_2} { \pi(x) \phi_{m}(x) \phi_{n}(x) dx }  =  \delta_{mn},
     \end{gathered}
\end{equation}
and corresponding to continuous eigenvalues
\begin{equation}
    \label{eq:Eq396}
    \begin{gathered}
        \int_{x_1}^{x_2} { \pi(\varepsilon) \phi( \lambda, \varepsilon ) \phi( \lambda^{'} , \varepsilon ) d\varepsilon }  =  \delta( \lambda - \lambda' ).
     \end{gathered}
\end{equation}
The Sturm-Liouville equation (\ref{eq:20180217489}) has at least one discrete eigenvalue equal to zero
$(\lambda_0 = 0)$ %
which corresponds to eigenfunction $\phi_0(x) = 1$.
Considering the discrete eigenvalues indexed as below,
\begin{equation}
    \label{eq:Eq435}
    \begin{gathered}
        0 = \lambda_0  \leq \lambda_1 \leq \lambda_2 \leq \cdots \leq \lambda_k.
    \end{gathered}
\end{equation}
It is assumed that equation
(\ref{eq:278}) %
has at least one more discrete eigenvalue
(other than $\lambda_0 = 0$)  %
denoted by
$\lambda_1$ %
which determines
the convergence rate of $Q(x,t)$ to its equilibrium point. %
If we consider the diffusion equation as a dynamical system, then
$\lambda_k$ %
are the Lyapunov exponents of the system and their inverse, i.e.,
$1/\lambda_1$ %
are the relaxation times of the system.

From (\ref{eq:Eq372}), it can be concluded that in the limit of $t \rightarrow \infty$, $p(t;x,y)$ converges to $\pi(y)$,
\begin{equation}
    \label{eq:Eq287}
    \begin{gathered}
        \lim_{t \rightarrow \infty} { p(t;x,y) }   =  \pi(y).
     \end{gathered}
\end{equation}
\section{Optimal Diffusion Processes}
\label{sec:OptimalDiffusionProcesses}

For a given $\pi(x)$, our aim is to obtain the optimal function for $\sigma^2(x)$
that maximizes
$\lambda_1$ %
for given $\pi(x)$, while satisfying the following constraint,
\begin{equation}
    \label{eq:Eq2.1}
    \begin{gathered}
        \int_{x_1}^{x_2} \frac{1}{2}\sigma^2(x) \pi(x) dx  =  \frac{1}{2}\widehat{\sigma}^{2}.
     \end{gathered}
\end{equation}
The function
$\frac{1}{2}\sigma^2(x)$  %
can be interpreted as a variance.
The constraint (\ref{eq:Eq2.1}) is equivalent to fix average variance constraint. %
Thus, our optimization problem can be interpreted as searching for the optimal variance function corresponding to minimum relaxation time, among all variance functions with the same average variance.

For a given variance function $(\sigma^{2}(x)/2)$ satisfying (\ref{eq:Eq2.1}), based on Rayleigh Theorem $\lambda_{1}$ can be written as below,
\begin{equation}
    \label{eq:Eq5}
    \begin{gathered}
        \max_{Q(x)} \;\; - \frac
        { \int_{x_1}^{x_2} Q(x) \frac{d}{d x} \left( \frac{1}{2}\sigma^{2}(x) \pi(x) \frac{d}{d x} Q(x) \right) dx }
        { \int_{x_1}^{x_2} \pi(x) \left( Q(x) \right)^2 dx }.
     \end{gathered}
\end{equation}
This is equivalent to the following,
\begin{equation}
    \label{eq:Eq6}
    \begin{gathered}
        \max_{Q(x)} \;\; \frac
        { \int_{x_1}^{x_2} \frac{1}{2}\sigma^{2}(x) \pi(x) \left( \frac{d}{d x} Q(x) \right)^2 dx }
        { \int_{x_1}^{x_2} \pi(x) \left( Q(x) \right)^2 dx }.
     \end{gathered}
\end{equation}
Therefore, using the min-max optimization
\cite{EkelandBook}, %
calculating optimal $\lambda_{1}$, reduces to the following problem,
\begin{equation}
    \label{eq:Eq7}
    \begin{aligned}
        \max\limits_{\sigma(x)} \;\; \min\limits_{Q(x)}
        \quad
        &
        \int_{x_1}^{x_2} \frac{1}{2}\sigma^{2}(x) \pi(x) \left( \frac{d}{d x} Q(x) \right)^2 dx,
        \\
        s.t.
        \quad
        &\int_{x_1}^{x_2} \pi(x) \left( Q(x) \right)^2 dx  =  1,
        \\& %
        \int_{x_1}^{x_2} \frac{1}{2}\sigma^{2}(x) \pi(x) dx  =  \frac{1}{2}\widehat{\sigma}^{2},
        \\& %
        \int_{x_1}^{x_2} \pi(x) Q(x)  dx  =  0,
    \end{aligned}
\end{equation}
with the following boundary conditions
\begin{equation}
    \label{eq:Eq8}
    \begin{gathered}
        \left. \frac{1}{2}\sigma^{2}(x) \pi(x) \frac{d}{d x} Q(x) \right|_{x=x_1}
        =
        \left. \frac{1}{2}\sigma^{2}(x) \pi(x) \frac{d}{d x} Q(x) \right|_{x=x_2}
        =
        0.
     \end{gathered}
\end{equation}
By introducing the relative Lagrange multipliers
$\left( \lambda, \nu, \gamma \right)$ %
the optimization problem in (\ref{eq:Eq7}) is transferred to the following optimization problem,
\begin{equation}
    \label{eq:Eq9}
    \begin{aligned}
        \max\limits_{\sigma(x)} \;\; \min\limits_{Q(x)}
        \quad
        &
        \int_{x_1}^{x_2} \frac{1}{2}\sigma^{2}(x) \pi(x) \left( \frac{d}{d x} Q(x) \right)^2 dx
        \\&
        +
        \lambda \left(  1  -  \int_{x_1}^{x_2} \pi(x) \left(  \pi(x) \right)^2 dx  \right)
        \\&
        +
        \nu \left(  \frac{1}{2}\widehat{\sigma}^{2}  -  \int_{x_1}^{x_2} \frac{1}{2}\sigma^{2}(x) \pi(x) dx  \right)
        \\&
        +
        \gamma \left(  \int_{x_1}^{x_2} \pi(x) Q(x)  dx  \right),
    \end{aligned}
\end{equation}
with the boundary conditions (\ref{eq:Eq8}).
To this aim, we have to set the variation of (\ref{eq:Eq9}) (with respect to $Q(x)$ and
$\frac{1}{2}\sigma^{2}(x)$ %
)
to zero, i.e.,
\begin{equation}
    \label{eq:Eq11}
    \begin{aligned}
        &
        -2 \int_{x_1}^{x_2}  \delta Q(x)  \left( \frac{d}{d x} \left(
        \frac{1}{2}\sigma^{2}(x) %
        \pi(x) \frac{d}{d x} Q(x) \right)
        \right.
        \left.
        +
        \lambda \pi(x) Q(x)  +  \gamma \pi(x)  \right)  dx
        \\&
        \quad
        +
        \int_{x_1}^{x_2}  \pi(x) \delta
        \frac{1}{2}\sigma^{2}(x) %
        \left( \left( \frac{d}{d x} Q(x) \right)^2 - \nu \right)  dx
        +
        \left.
        \frac{1}{2}\sigma^{2}(x) %
        \pi(x) \delta Q(x) \frac{d}{d x} Q(x) \right|_{x=x_i}  =  0,
     \end{aligned}
\end{equation}
for $i=1, 2$.
Due to the boundary conditions (\ref{eq:Eq8}),
the last term
$\left. \frac{1}{2}\sigma^{2}(x) \pi(x) \delta Q(x) \frac{d}{d x} Q(x) \right|_{x=x_i}  =  0$
for $i=1, 2$.
The optimal $Q(x)$ and $\lambda$ are $\phi_1(x)$ and $\lambda_{1}$, respectively.
Thus,
the %
variation of Lagrangian leads to the following,
\begin{equation}
    \label{eq:Eq12}
    \begin{gathered}
        \frac{d}{d x} \left( \frac{1}{2}\sigma^{2}(x) \pi(x) \frac{d}{d x} \phi_1(x) \right)
        +
        \lambda_1 \pi(x) \phi_1(x)
        +
        \gamma \pi(x)
        =
        0,
     \end{gathered}
\end{equation}
\begin{equation}
    \label{eq:Eq13}
    \begin{gathered}
        \left(  \frac{d}{d x} \phi_1(x)  \right)^2
        =
        \nu.
     \end{gathered}
\end{equation}
Integrating Equation (\ref{eq:Eq12}) from $x_1$ to $x_2$ and considering the boundary conditions
(\ref{eq:Eq8}), %
and the constraints in problem (\ref{eq:Eq7}),
then 
it can be concluded that $\int_{x_1}^{x_2}{\gamma \pi(x)} = 0$ and considering the fact that $\int_{x_1}^{x_2}{\pi(x)} = 1$,
then we can be concluded that $\gamma = 0$.
From (\ref{eq:Eq13}), it is obvious that
\begin{equation}
    \label{eq:Eq609}
    \begin{gathered}
        \frac{d}{d x} \phi_1(x)
        =
        \sqrt{ \nu }.
     \end{gathered}
\end{equation}
Substituting (\ref{eq:Eq609}) in (\ref{eq:Eq12}), %
and comparing the resultant equation with
(\ref{eq:20170727442}) %
the following can be concluded for
$\mu(x)$, %
\begin{equation}
    \label{eq:Eq634}
    \begin{gathered}
        \mu(x)  =  \frac{-\lambda_1}{\sqrt{\nu}}\phi_1(x).
     \end{gathered}
\end{equation}
Substituting
$\gamma = 0$ %
in (\ref{eq:Eq12}), and multiplying the resultant by $\phi_1(x)$ and integrating it from $x_1$ to $x_2$, we have
\begin{equation}
    \label{eq:201803031357}
    \begin{gathered}
        \int_{x_1}^{x_2} \frac{d}{d x} \left( \frac{1}{2}\sigma^{2}(x) \pi(x) \frac{d}{d x} \phi_1(x) \right) \phi_1(x)  dx
        +
        \int_{x_1}^{x_2} \lambda_1 \pi(x) \phi_1^{2}(x) dx
        =
        0.
     \end{gathered}
\end{equation}
This is equivalent to the following
\begin{equation}
    \label{eq:201803031378}
    \begin{gathered}
        - \int_{x_1}^{x_2} \frac{1}{2}\sigma^{2}(x) \pi(x)  \left( \frac{d}{d x} \phi_1(x) \right)^{2} dx
        +
        \lambda_1  \int_{x_1}^{x_2} \pi(x) \phi_1^{2}(x) dx
        =
        0.
     \end{gathered}
\end{equation}
Using (\ref{eq:Eq386}), (\ref{eq:Eq2.1}) and (\ref{eq:Eq13}), from (\ref{eq:201803031378}), the following can be concluded for $\lambda_{1}$
\begin{equation}
    \label{eq:201802171163}
    \begin{gathered}
        \lambda_{1}
        =
        \nu \frac{ \widehat{\sigma}^{2} }{ 2 }.
     \end{gathered}
\end{equation}

\section{Main Results}
\label{sec:MainResults}
In the following theorem, we state the main results of this paper deducted based on Section
\ref{sec:OptimalDiffusionProcesses}
and Appendix 
\ref{Sec:DetailedSolution}.
\begin{theorem} %
    \label{Theorem:Mainresults}
    For given stationary distribution $\pi(x)$ and average variance $\widehat{\sigma}^{2}/2$,
    the optimal results corresponding to the diffusion process with fastest convergence rate are as below,
    \begin{equation}
        \label{eq:201702181190}
        \begin{gathered}
            \lambda_1 = \frac{ \widehat{\sigma}^{2} }{ 2 \left( m_{2} - m_{1}^{2} \right) },
            \;\;
            \tau = \frac{1}{\lambda_1} = \frac{ 2 \left( m_{2} - m_{1}^{2} \right) }{ \widehat{\sigma}^{2} },
         \end{gathered}
    \end{equation}
    \begin{equation}
        \label{eq:Eq760}
        \begin{gathered}
            \phi_1(x)
            =
            \frac{ x - m_1 }{ \sqrt{m_2 - m_1^2} },
         \end{gathered}
    \end{equation}
    \begin{equation}
        \label{eq:201803031404} %
        \begin{gathered}
            \mu(x)
            =
            \frac{ \widehat{\sigma}^{2} \left( m_1 - x \right) }{ 2 \left( m_2 - m_1^2 \right) },
         \end{gathered}
    \end{equation}
    \begin{equation}
        \label{eq:201803021411}
        \begin{gathered}
            \frac{1}{2}\sigma^{2}(x)=
            \frac{1}{\pi(x)} \left( \frac{ \widehat{\sigma}^{2} }{ 2 \left( m_2 - m_1^2 \right) } \left( \left( m_1 - x \right) \Pi(x)  +  \int_{x_1}^{x}\Pi(z) dz \right) \right),
         \end{gathered}
    \end{equation}
    where $\lambda_{1}$ and $\tau = 1 / \lambda_{1}$ are the optimal convergence rate and the optimal relaxation time of the diffusion process.
    $\phi_1(x)$ is the optimal eigenfunction corresponding to the optimal $\lambda_{1}$.
    $m_1$ and $m_2$ are first and second moments of the stationary distribution $\pi(x)$.
    $\mu(x)$ and $\frac{1}{2}\sigma^{2}(x)$ are the optimal drift and variance functions.
    $\Pi(x)$ is the cumulative distribution of $\pi(x)$ (i.e., $\Pi(x) = \int_{x_{1}}^{x} \pi(z) dz$).
    From
    (\ref{eq:201702181190}),
    it can be concluded that among diffusion processes where the average of their variance function is equal to $\widehat{\sigma}^{2}/2$, the optimal achievable convergence rate is as in
    (\ref{eq:201702181190}).
    Thus, $\frac{ \widehat{\sigma}^{2} }{ 2 \left( m_{2} - m_{1}^{2} \right) }$ is an upper bound on the convergence rate of the diffusion process with given average variance $\widehat{\sigma}^{2}/2$ and stationary distribution $\pi(x)$.
\end{theorem}
In the following lemmas, we have proved the convexity of the optimal convergence rate $(\lambda_1)$ as a function of the stationary distribution $(\pi(x))$.
Also,
we have shown that variance function obtained in
(\ref{eq:201803021411}) %
is a positive function in the interval $[ x_1, x_2 ]$ and its weighted average (with weights $\pi(x)$) is equal to $\frac{ \widehat{\sigma}^{2} }{ 2 }$.
\begin{lemma}
    The optimal relaxation time $(\tau)$ is a concave function of the stationary distribution $(\pi(x))$.
\end{lemma}
\begin{proof}
Considering $m$ diffusion processes
with stationary distributions $\pi_i(x)$ for $i=1, ..., m$ which are defined over the same interval
and
have the same average variance $\left( \frac{\widehat{\sigma}^{2}}{2} \right)$,
the convex combination of their stationary distributions is as below,
\begin{equation}
    \label{eq:201803031597}
    \begin{gathered}
        \pi(x)=\sum\nolimits_{i=1}^m p_i\pi_i(x),
        \quad
        \Pi(x)=\sum\nolimits_{i=1}^m p_i\Pi_i(x),
    \end{gathered}
\end{equation}
where $p_i \geq 0$, for $i=1, ..., m$ and $\sum_{i=1}^{m} p_i = 1$.
For the mean and variance of the resultant stationary distribution, we have
\begin{equation}
    \label{eq:201803031607}
    \begin{gathered}
        m_1=\sum\nolimits_{i=1}^m p_im_1(i),
    \end{gathered}
\end{equation}
\begin{equation}
    \label{eq:201803031622}
    \begin{gathered}
        m_2 - m_1^2 = \sum\nolimits_{i=1}^m p_i((m_1-m_1(i))^2+(m_2(i)-m_1(i)^2)).
    \end{gathered}
\end{equation}
Therefore, $ m_2 - m_1^2$ is a concave function, i.e., we have
\begin{equation}
    \label{eq:201803031621}
    \begin{gathered}
        m_2-m_1^2 \geq \sum\nolimits_{i=1}^m p_i(m_2(i)-m_1(i)^2).
    \end{gathered}
\end{equation}
Hence, from (\ref{eq:201702181190}) it can be concluded that the optimal relaxation time $\tau$ is a concave  function of stationary distribution, i.e.,
\begin{equation}
    \begin{gathered}
        \tau \left( \sum\nolimits_{i=1}^{m} p_{i} \pi_{i}\left( x \right) \right)
        \geq
        \sum\nolimits_{i=1}^{m} p_i \tau \left( \pi_{i}\left( x \right) \right),
    \end{gathered}
\end{equation}
where $\tau\left( \pi_{i}\left( x \right) \right)$ is the optimal relaxation time corresponding to the $i$-th stationary distribution (i.e., $\pi_{i}(x)$),
and
$\tau \left( \sum\nolimits_{i=1}^{m} p_{i} \pi_{i}\left( x \right) \right)$
is the optimal relaxation time corresponding to the equilibrium distribution $\pi\left( x \right)  =  \sum\nolimits_{i=1}^{m} p_{i} \pi_{i}\left( x \right)$.
\end{proof}
\begin{lemma}
    The variance function $\frac{ \sigma^2(x) }{2}$ in
    (\ref{eq:201803021411})    %
    is a positive function for all $x$ in the interval $[ x_1, x_2 ]$.
\end{lemma}
{
\begin{proof}
    Based on
    (\ref{eq:201803021411}),    %
    we define $V(x)$ as below,
    \begin{equation}
        \label{eq:201802241389}
        \begin{gathered}
            V \left( x \right)
            =
            \frac{1}{\pi(x)} \int_{x_1}^{x} \frac{ \widehat{\sigma}^{2} \left( m_1 - z \right) }{ 2 \left( m_2 - m_1^2 \right) }  \pi(z)  dz.
         \end{gathered}
    \end{equation}
    From (\ref{eq:201802241389}), it can be concluded that $\pi\left( x \right) V(x) = 0$ for $x = x_1, x_2$.
    Since $\pi(x)$ is a positive function and $m_{1}$ is the weighted average of $x$ (with weight $\pi(x)$) in the range $[ x_1, x_2 ]$, then it is obvious that $x_1 < m_1 < x_2$.
    Defining
    $G(x)  =  \frac{d}{dx}  \left( \pi(x) V(x) \right)  =  \frac{ \widehat{\sigma}^{2} / 2 }{ m_{2} - m_{1}^{2} } \left( m_1 - x \right) \pi(x)$,
    and considering $\pi\left( x_1 \right) V\left( x_1 \right) = 0$,
    we have
    \begin{equation}
        \label{eq:201802243215}
        \begin{gathered}
            \int_{x_1}^{x} G(y) dy
            =
            \pi\left( x \right) V(x)
            -
            \pi\left( x_1 \right) V\left( x_1 \right)
            =
            \pi\left( x \right) V(x).
        \end{gathered}
    \end{equation}
    Since $G(x)$ for $x \in [ x_1, m_1 ]$ is positive, then from $(\ref{eq:201802243215})$ it can be concluded that $\int_{x_1}^{x} G(y) dy > 0$ for $x \in [ x_1, m_1 ]$ and consequently the variance function $V(x)$ for $x \in [ x_1, m_1 ]$ is a positive function.
    For $x \in [ m_1, x_2 ]$, considering
    $\pi\left( x_2 \right) V\left( x_2 \right) = 0$,
    we have
    \begin{equation}
        \label{eq:201802243235}
        \begin{gathered}
            \int_{x}^{x_2} G(y) dy
            =
            \pi\left( x_2 \right) V\left( x_2 \right)
            -
            \pi\left( x \right) V\left( x \right)
            =
            - \pi\left( x \right) V\left( x \right).
        \end{gathered}
    \end{equation}
    Since $G(x)$ for $x \in [ m_1, x_2 ]$ is negative, then from $(\ref{eq:201802243235})$, it can be concluded that $\int_{x}^{x_2} G(y) dy < 0$ for $x \in [ m_1, x_2 ]$ and consequently the variance function $V(x)$ for $x \in [ m_1, x_2 ]$ is a positive function.
\end{proof}
}

\begin{lemma}
    The weighted average of the variance function $\frac{ \sigma^2(x) }{2}$ in
    (\ref{eq:201803021411})    %
    (with weights $\pi(x)$) is equal to $\frac{ \widehat{\sigma}^{2} }{ 2 }$.
\end{lemma}
\begin{proof}
    Considering $V(x)$ defined in (\ref{eq:201802241389}), we have
    \begin{equation}
        \label{eq:201802243156}
        \begin{aligned}
            \int_{x_1}^{x_2} V(x)  \pi(x) dx
            &
            =
            \frac{ \widehat{\sigma}^{2} / 2 }{ m_{2} - m_{1}^{2} }  \int_{x_1}^{x_2}  dx \int_{x_1}^{x} \left( m_{1} - \xi \right) \pi(\xi) d\xi
            \\
            &
            =
            \frac{ \widehat{\sigma}^{2} / 2 }{ m_{2} - m_{1}^{2} }  \int_{x_1}^{x_2} d \left( x \int_{x_1}^{x} \left( m_{1} - \xi \right) \pi(\xi) d\xi \right)
            -
            \frac{ \widehat{\sigma}^{2} / 2 }{ m_2 - m_1^2 } \int_{ x_1 }^{ x_2 } x d \left( \int_{ x_1 }^{ x } \left( m_1 - \xi \right) \pi(\xi) d\xi \right)
            \\
            &
            =
            \frac{ \widehat{\sigma}^{2} / 2 }{ m_2 - m_1^2 }
            \left(
                \int_{x_1}^{x_2} \left( m_1 - \xi \right)  \pi(\xi) d\xi
                -
                \int_{x_1}^{x_2} x \left( m_1 - x \right) \pi(x) dx
            \right)
            \\
            &
            =
            \left( \frac{ \widehat{\sigma}^{2} / 2 }{ m_2 - m_1^2 } \right) \left(  m_2 - m_1^2  \right)
            \\
            &
            =
            \frac{ \widehat{\sigma}^{2} }{ 2 }.
        \end{aligned}
    \end{equation}
\end{proof}

\section{Examples of Optimal Diffusion Processes}
\label{sec:examples}

In this section, we report the optimal results for three types of diffusion processes,
namely
Pearson class of diffusion processes,
diffusion processes with polynomial variance functions of degree higher than two,
and
the diffusion process with Hyperexponential distribution as the stationary distribution which is an example of the class of diffusion processes with non-polynomial variance function.

\subsection{Pearson Class of Diffusion Processes}

Diffusion processes
whose %
drift and variance functions are polynomials of degrees at most one and two, respectively, are referred to as the Pearson class of diffusion processes,
since their stationary distribution satisfies the famous Pearson equation (\ref{eq:20170727442}).
Thus, they are all optimal in terms of the convergence rate $(\lambda_1)$.
In Table \ref{tab:Table2}, we have provided different attributes of
some of the well-known types of
Pearson diffusion processes, including
the stationary distributions,
along with their first and second moments,
the spectrum of the corresponding infinitesimal generators (and $\lambda_1$ which determines the optimal convergence rate),
the corresponding differential equations
and
the optimal variance functions (obtained using
(\ref{eq:201803021411}))    %
along with its average and the interval.
The Pearson diffusion processes included in Table \ref{tab:Table2} are also known as Pearson diffusion processes of the Hypergeometric type, since their corresponding eigenfunctions are written in terms of either the Hypergeometric functions or the confluent Hypergeometric functions.
Among the Pearson diffusion processes included in Table \ref{tab:Table2},
for
Hypergeometric, Jacobi, CIR, Ornstein-Uhlenbeck processes,
all moments of the stationary distributions have finite values,
while for
Fisher-Snedecor, Student and Reciprocal Gamma processes,
only finite number of moments of the stationary distributions have finite values.
The stationary distributions of
Fisher-Snedecor, Student and Reciprocal Gamma processes,
are also known as heavy-tailed distributions, which are used in different applications including communication networks, insurance industry and modeling the price of risky assets \cite{LEONENKO201030,Peskir2006,martinbibby2005}.
\begin{table*}[t]
\centering
    \caption{Optimal Pearson diffusion processes of the hypergeometric type with variance $(\sigma^{2}(x)/2)$ of at most degree two.}
    \label{tab:Table2}    %
    \begin{tabular}{|c|c|c|c|c|c|c|} \hline
    {
        \hspace{-12pt}
        \begin{minipage}{60pt}
            \centering
            \footnotesize{Eqi. Dist.: $\boldsymbol{\pi}(x)$}
            \\
            \vspace{1pt}
            \footnotesize{$x$ Interval}
            \vspace{1pt}
        \end{minipage}
        \hspace{-12pt}
    } %
    & { \hspace{-12pt} \begin{minipage}{65pt} \centering \small{$m_1$,} \\  \small{$m_2 - m_1^2$} \end{minipage}  \hspace{-12pt} }  %
    & { \hspace{-12pt} \begin{minipage}{25pt} \centering \small{$\lambda_n$,} \\ \small{$\lambda_1$} \end{minipage}   \hspace{-12pt} }  %
    & { \hspace{-12pt} \small{$\boldsymbol{\phi}_n\left(x\right)$} \hspace{-12pt} }  %
    & {
        \hspace{-12pt}
        \begin{minipage}{130pt} \centering
            Process Name,
            \\
            Differential Equation
        \end{minipage}
        \hspace{-12pt}
    } %
    & { \hspace{-12pt}
            \begin{minipage}{40pt}
                \centering
                \vspace{1pt}
                    \scriptsize{$\sigma^{2}(x) / 2$,}
                \vspace{2pt}
                \\
                \scriptsize{$\widehat{\sigma}^{2} / 2$,}
                \\
                \scriptsize{$\mu(x)$,}
            \end{minipage}
        \hspace{-12pt}
      } %
    \\ \hline
    {   \hspace{-12pt}
        \begin{minipage}{60pt}
            \centering
            \scriptsize{Beta: $\frac{x^\alpha (1-x)^\beta}{B(\alpha+1,\beta+1)}$}
            \\
            \scriptsize{$-1<\alpha,\beta$,}
            \\
            \scriptsize{$x \in [0,1]$}
        \end{minipage}
        \hspace{-12pt}
    }  %
    & { \hspace{-12pt} \begin{minipage}{65pt} \centering \footnotesize{$\frac{\alpha+1}{\alpha+\beta+2}$,} \\  \footnotesize{$\frac{(\alpha+1) (\beta+1)}{(\alpha+\beta+2)^2(\alpha+\beta+3)}$} \end{minipage}  \hspace{-12pt} }  %
    & { \hspace{-12pt} \begin{minipage}{55pt} \centering \footnotesize{$n(n+\alpha+\beta+1)$,} \\ \footnotesize{$\alpha+\beta+2$} \end{minipage}  \hspace{-12pt} }  %
    & { \hspace{-12pt} \begin{minipage}{120pt} \centering \scriptsize{Hypergeometric Poly.:} \\  \footnotesize{${}_2F_1(-n, n+\alpha+\beta+1;\alpha+1 ;x)$} \end{minipage}  \hspace{-12pt} }  %
    & { \hspace{-12pt} \begin{minipage}{130pt} \centering
        \vspace{1pt}
        \footnotesize{Hypergeometric Process,}
        \\
        \vspace{1pt}
        \footnotesize{$x \left( 1 - x \right) \frac{d^2\phi_n(x)}{dx^2}$ $+$
        $\left[ \left( \alpha + 1 \right) - \left( \alpha + \beta + 2 \right) x \right] \frac{d\phi_n(x)}{dx}$ $+$
        }
        \\
        \footnotesize{
        $n \left( n + \alpha + \beta + 1 \right)$ $\phi_n(x)$
        $=$
        $0$
        }
        \end{minipage}  \hspace{-12pt}
        } %
    &{ \hspace{-12pt} \begin{minipage}{65pt} \centering
        \vspace{1pt}
        \scriptsize{ $x(1-x)$, }
        \\
        \vspace{1pt}
        \scriptsize{ $\frac{(\alpha+1)(\beta+1)}{(\alpha+\beta+3)(\alpha+\beta+2)}$, }
        \\
        \vspace{1pt}
        \resizebox{.97\hsize}{!}{$\alpha+1 - x(\alpha+\beta+2)$,}
        \end{minipage}  \hspace{-12pt}
    } %
    \\ \hline
    {   \hspace{-11pt}
        \begin{minipage}{70pt}
            \centering
            \vspace{1pt}
            \footnotesize{Jacobi:}
            \\
            \footnotesize{$\frac{(1-x)^\alpha (1+x)^\beta}{2^{\alpha+\beta+1}B(\alpha+1,\beta+1)}$}
            \\
            \footnotesize{$-1<\alpha,\beta$,}
            \\
            \footnotesize{$x \in [-1, 1]$}
        \end{minipage}
        \hspace{-11pt}
    }  %
    &{  \hspace{-12pt}
        \begin{minipage}{65pt}
            \centering
            \footnotesize{$\frac{\beta-\alpha}{\alpha+\beta+2} $,}
            \\
            \footnotesize{$\frac{4(\alpha+1) (\beta+1)}{(\alpha+\beta+2)^2(\alpha+\beta+3)}$}
        \end{minipage}
        \hspace{-12pt}
    }  %
    & { \hspace{-10pt}
        \begin{minipage}{55pt}
            \centering
            \footnotesize{$n(n+\alpha+\beta+1)$,}
            \\
            \footnotesize{$\alpha+\beta+2$}
        \end{minipage}
        \hspace{-10pt}
    }  %
    & { \hspace{-12pt}
        \begin{minipage}{120pt}
            \centering
            \scriptsize{Jacobi Poly.: $P_n^{\alpha,\beta}(x)$}
            \\
            \scriptsize{${}_2F_1(-n, n+\alpha+\beta+1;\alpha+1 ;\frac{1-x}{2})$}
        \end{minipage}
        \hspace{-12pt}
    }  %
    &{  \hspace{-11pt}
        \begin{minipage}{130pt}
            \centering
            \footnotesize{Jacobi Process,}
            \\
            \scriptsize{$(1-x^2)\frac{d^2\phi_n(x)}{dx^2}+[(\beta-\alpha)-(\alpha+\beta+2)x]\frac{d\phi_n(x)}{dx}+n(n+\alpha+\beta+1)\phi_n(x)=0$}
        \end{minipage}
        \hspace{-11pt}
    } %
    & { \hspace{-12pt} \begin{minipage}{65pt} \centering
        \vspace{1pt}
        \footnotesize{ $1-x^{2}$, }
        \\
        \vspace{1pt}
        \footnotesize{ $\frac{4(\alpha+1)(\beta+1)}{(\alpha+\beta+3)(\alpha+\beta+2)}$, }
        \\
        \resizebox{.98\hsize}{!}{$\beta-\alpha-(\alpha+\beta+2)x$,}
        \end{minipage}  \hspace{-12pt}
    } %
    \\ \hline
    {
        \hspace{-11pt}
        \begin{minipage}{70pt}
            \centering
            \scriptsize{Gamma: $\frac{x^\alpha e^{-x}}{\Gamma(\alpha+1)}$,}
            \\  \vspace{1pt}
            \scriptsize{$-1<\alpha$,}
            \\ \vspace{1pt}
            \scriptsize{$x \in [0, \infty)$}
            \vspace{1pt}
        \end{minipage}
        \hspace{-11pt}
    }  %
    &{  \hspace{-12pt}
        \begin{minipage}{65pt}
            \centering
            \footnotesize{$\alpha+1$,}
            \\
            \footnotesize{$\alpha+1$}
        \end{minipage}
        \hspace{-12pt}
    }  %
    & { \hspace{-10pt}
        \begin{minipage}{55pt}
            \centering
            \footnotesize{$n$,}
            \\
            \footnotesize{$1$}
        \end{minipage}
        \hspace{-10pt}
    }  %
    & { \hspace{-12pt}
        \begin{minipage}{120pt}
            \centering
            \scriptsize{Laguerre Poly.: $L_n^\alpha(x)=$}
            \\
            \scriptsize{${}_1F_1(-n;\alpha+1;x)$}
        \end{minipage}
        \hspace{-12pt}
    }  %
    &{  \hspace{-11pt}
        \begin{minipage}{130pt}
            \centering
            \footnotesize{CIR Process,}
            \\
            \scriptsize{$x\frac{d^2\phi_n(x)}{dx^2}+(\alpha+1-x)\frac{d\phi_n(x)}{dx}+n\phi_n(x)=0$}
        \end{minipage}
        \hspace{-11pt}
    } %
    & { \hspace{-12pt} \begin{minipage}{65pt} \centering
        \vspace{1pt}
        \scriptsize{ $x$, }
        \\
        \vspace{1pt}
        \scriptsize{ $\alpha + 1$, }
        \\
        \vspace{1pt}
        \scriptsize{$\alpha+1-x$,}
        \end{minipage}  \hspace{-12pt}
    } %
    \\ \hline
    {   \hspace{-11pt}
        \begin{minipage}{70pt}
            \centering
            \vspace{1pt}
            \scriptsize{Normal: $\frac{e^{-\frac{(x-x_0)^2}{2\sigma^2}}}{\sqrt{2\pi\sigma^2}}$,}
            \\
            \footnotesize{$x \in \left( -\infty, \infty \right)$}
        \end{minipage}
        \hspace{-11pt}
    }  %
    &{  \hspace{-12pt}
        \begin{minipage}{65pt}
            \centering
            \footnotesize{$x_{0}$,}
            \\
            \footnotesize{$\sigma^{2}$}
        \end{minipage}
        \hspace{-12pt}
    }  %
    & { \hspace{-10pt}
        \begin{minipage}{55pt}
            \centering
            \footnotesize{$n$,}
            \\
            \footnotesize{$1$}
        \end{minipage}
        \hspace{-10pt}
    }  %
    & { \hspace{-12pt}
        \begin{minipage}{120pt}
            \centering
            \scriptsize{Hermite Poly.: $H_n(\frac{x-x_0}{\sigma})=$}
            \\
            \scriptsize{$(\frac{x-x_0}{\sigma})^n{}_2F_0(\frac{-n}{2},\frac{1-n}{2}; \frac{-\sigma^2}{(x-x_0)^2})$}
        \end{minipage}
        \hspace{-12pt}
    }  %
    &{  \hspace{-11pt}
        \begin{minipage}{130pt}
            \centering
            \footnotesize{Ornstein-Uhlenbeck Process,}
            \vspace{1pt}
            \\
            \scriptsize{$\frac{d^2\phi_n(x)}{dx^2}-x\frac{d\phi_n(x)}{dx}+n\phi_n(x)=0$}
        \end{minipage}
        \hspace{-11pt}
    } %
    & { \hspace{-12pt} \begin{minipage}{65pt} \centering
        \vspace{1pt}
        \footnotesize{ $1$, }
        \\
        \vspace{1pt}
        \footnotesize{ $1$, }
        \\
        \vspace{1pt}
        \footnotesize{$-x$}
        \end{minipage}  \hspace{-12pt}
    } %
    \\ \hline
    {   \hspace{-11pt}
        \begin{minipage}{70pt}
            \centering
            \resizebox{.99\hsize}{!}{Cauchy: $\frac{ \left( 1 + x^2 \right)^{-(\alpha + \frac{1}{2})} }{ B\left( \alpha , \frac{1}{2} \right) }$}
            \\
            \scriptsize{$2\leq\alpha$,}
            \\
            \scriptsize{$x \in (-\infty, \infty)$}
        \end{minipage}
        \hspace{-11pt}
    }  %
    &{  \hspace{-12pt}
        \begin{minipage}{65pt}
            \centering
            \footnotesize{$0$,}
            \\
            \footnotesize{$\frac{1}{2(\alpha-1)}$}
        \end{minipage}
        \hspace{-12pt}
    }  %
    & { \hspace{-10pt}
        \begin{minipage}{55pt}
            \centering
            \footnotesize{$n(2\alpha-n)$ for $n=0,...,\lfloor \alpha \rfloor$,}
            \\
            \footnotesize{$2\alpha-1$}
        \end{minipage}
        \hspace{-10pt}
    }  %
    & { \hspace{-12pt}
        \begin{minipage}{120pt}
            \centering
            \scriptsize{$\sqrt{\frac{(\alpha-n)\Gamma(\alpha)}{\sqrt{\pi}n!\Gamma(\alpha+\frac{1}{2})}}2^{\alpha-n}\Gamma(\alpha-n+\frac{1}{2})$}
            \\
            \scriptsize{$(1+x^2)^{\alpha+\frac{1}{2}}\frac{d^n}{dx^n}(1+x^2)^{n-\alpha-\frac{1}{2}}$}
        \end{minipage}
        \hspace{-12pt}
    }  %
    &{  \hspace{-11pt}
        \begin{minipage}{130pt}
            \centering
            \footnotesize{Student Process,}
            \\
            \scriptsize{$(1+x^2)\frac{d^2\phi_n(x)}{dx^2}  +  \left( 1 - 2\alpha \right)x \frac{d\phi_n(x)}{dx}  +  n\left( 2 \alpha - n \right) \phi_n(x)  =  0$,}
        \end{minipage}
        \hspace{-11pt}
    } %
    & { \hspace{-12pt} \begin{minipage}{65pt} \centering
        \vspace{1pt}
        \footnotesize{ $1 + x^{2}$, }
        \\
        \vspace{1pt}
        \footnotesize{ $\frac{ \alpha - 1 }{ \alpha - \frac{1}{2} }$, }
        \\
        \vspace{1pt}
        \footnotesize{$(1-2\alpha)x$}
        \end{minipage}  \hspace{-12pt}
    } %
    \\ \hline
    {   \hspace{-11pt}
        \begin{minipage}{70pt}
            \centering
            \scriptsize{Inverse Gamma:}
            \\
            \scriptsize{$\frac{ x^{-(2\alpha+1)}e^{-\frac{1}{x} }}{\Gamma(2\alpha) }$,}
            \\
            \vspace{1pt}
            \tiny{$2\leq\alpha$,}
            \\
            \vspace{1pt}
            \tiny{$x \in [0, \infty)$}
        \end{minipage}
        \hspace{-11pt}
    }  %
    &{  \hspace{-12pt}
        \begin{minipage}{65pt}
            \centering
            \footnotesize{$\frac{1}{2\alpha-1}$,}
            \\
            \footnotesize{$\frac{1}{2(\alpha-1)(\alpha-1)^2}$}
        \end{minipage}
        \hspace{-12pt}
    }  %
    & { \hspace{-10pt}
        \begin{minipage}{55pt}
            \centering
            \footnotesize{$n(2\alpha-n)$ for $n=0,...,\lfloor \alpha \rfloor$,}
            \\
            \footnotesize{$2\alpha-1$}
        \end{minipage}
        \hspace{-10pt}
    }  %
    & { \hspace{-12pt}
        \begin{minipage}{120pt}
            \centering
            \scriptsize{$(-1)^n \sqrt{\frac{2(\alpha-n)\Gamma(2\alpha)}{n!\Gamma(2\alpha+1-n)}}$}
            \\
            \scriptsize{$x^{2\alpha+1}e^{\frac{1}{x}}\frac{d^n}{dx^n}(x^{2n-2\alpha-1}e^{-\frac{1}{x}})$}
        \end{minipage}
        \hspace{-12pt}
    }  %
    &{  \hspace{-11pt}
        \begin{minipage}{130pt}
            \centering
            \footnotesize{Reciprocal Gamma Process,}
            \\
            \resizebox{.99\hsize}{!}{$x^2 \frac{d^2\phi_n(x)}{dx^2}   +   \left( 1 - \left( 2\alpha + 1 \right) x \right) \frac{d\phi_n(x)}{dx}$}
            \\
            \resizebox{.62\hsize}{!}{$+ n\left( 2 \alpha - n \right) \phi_n(x)   =   0$}
        \end{minipage}
        \hspace{-11pt}
    } %
    & { \hspace{-12pt} \begin{minipage}{65pt} \centering
        \vspace{1pt}
        \footnotesize{ $x^{2}$, }
        \\
        \vspace{1pt}
        \footnotesize{ $\frac{ 2 \alpha - 1 }{ 2 \alpha - 2 }$, }
        \\
        \scriptsize{$1-(2\alpha+1)x$}
        \end{minipage}  \hspace{-12pt}
    } %
    \\ \hline
    {   \hspace{-11pt}
        \begin{minipage}{70pt}
            \centering
            \resizebox{.99\hsize}{!}{Snedecor-Fisher's $F$-dist.:}
            \\
            \vspace{2pt}
            \resizebox{0.99\hsize}{!}{$\frac{  \left( \frac{\nu_1}{\nu_2} \right)^{\frac{\nu_1}{2}} x^{\frac{\nu_1}{2} - 1}  }{  B \left( \frac{\nu_1}{2} , \frac{\nu_2}{2} \right) \left( 1 + \frac{\nu_1}{\nu_2}x \right)^{ \frac{\nu_1 + \nu_2}{2} }  }$}
            \\
            \vspace{1pt}
            \tiny{$x \in [0, \infty)$}
        \end{minipage}
        \hspace{-11pt}
    }  %
    &{  \hspace{-12pt}
        \begin{minipage}{65pt}
            \centering
            \footnotesize{$\frac{ \nu_{2} }{ \nu_{2}-2 }$,}
            \\
            \footnotesize{$\frac{  2 \nu_{2}^{2} ( \nu_{2} + \nu_{1} - 2 )  }{  \nu_{1} ( \nu_{2} - 2 )^{2} ( \nu_{2} - 4 )  }$}
        \end{minipage}
        \hspace{-12pt}
    }  %
    & { \hspace{-10pt}
        \begin{minipage}{55pt}
            \centering
            \resizebox{0.99\hsize}{!}{$\frac{\nu_1}{2\nu_2}n(6+\nu_2-2n)$,}
            \\
            \footnotesize{$\nu_{1} \left(  \frac{ 2 }{ \nu_{2} }  +  \frac{ 1 }{ 2 }  \right)$}
        \end{minipage}
        \hspace{-10pt}
    }  %
    & { \hspace{-12pt}
        \begin{minipage}{120pt}
            \centering
            \resizebox{0.99\hsize}{!}{$\nu_{1} \nu_{2}^{n} 2^{n-1} {}_{2}F_{1}\left( -n, n - \frac{ \nu_{2} }{ 2 }; \frac{ \nu_{1} }{ 2 }; -\frac{ \nu_{1} }{ \nu_{2} }x  \right)$}
        \end{minipage}
        \hspace{-12pt}
    }  %
    &{  \hspace{-11pt}
        \begin{minipage}{130pt}
            \centering
            \vspace{1pt}
            {Fisher Snedecor Diffusion}
            \vspace{1pt}
            \\
            \scriptsize{$x \left( 1 + \frac{\nu_1}{\nu_2}x \right) \frac{d^{2} \phi_{n}}{dx^{2}}
            +
            \left(  \frac{\nu_1}{2} - 1 - \nu_1 \left( \frac{2}{\nu_2} + \frac{1}{2} \right) x  \right) \frac{d \phi_{n}}{d x}
            +
            \frac{\nu_1}{2\nu_2}n(6+\nu_2-2n) \phi_n  =  0$}
        \end{minipage}
        \hspace{-11pt}
    } %
    & { \hspace{-12pt} \begin{minipage}{65pt} \centering
        \scriptsize{ $x \left( 1 + \frac{ \nu_{1} }{ \nu_{2} } x \right)$, }
        \\
        \scriptsize{ $\frac{ \nu_{2} \left( \nu_{2} - 4 \right) }{ \nu_{1} + \nu_{2} - 2 }$, }
        \\
        \resizebox{.98\hsize}{!}{$\frac{\nu_1}{2} - 1 - \left( \frac{2\nu_1}{\nu_2} + \frac{\nu_1}{2} \right) x$}
        \end{minipage}  \hspace{-12pt}
    } %
    \\ \hline
    \end{tabular}
\end{table*}

\subsection{Diffusion Processes with Polynomial Variance Functions of Degree Higher than Two}

In this subsection, we consider a diffusion process where its variance function is a polynomial function with degree higher than two.
Considering the following stationary distribution,
\begin{equation}
    \label{eq:201803043479}
    \begin{gathered}
        \pi(x)
        =
        \frac
        {  x^{\alpha - 1} \left( 1 - x \right)^{\beta-1} \left(1 - \alpha x \right)^{ - \left( \alpha + \beta + 1 \right) }  }
        {  B\left( \alpha, \beta \right)  {}_{2}F_{1}\left( \alpha+\beta+1, \alpha, \alpha+\beta, a \right)  },
    \end{gathered}
\end{equation}
with $|a| < 1$, $0 < \alpha, \beta$,
the moments of the stationary distribution can be written as below
\begin{equation}
    \label{eq:201803043516}
    \begin{gathered}
        m_{1}
        =
        \frac
        {   B\left( \alpha+1, \beta \right)  {}_{2}F_{1}\left( \alpha + \beta + 2, \alpha + 1, \alpha + \beta + 1, a \right)        }
        {   B\left( \alpha, \beta \right) {}_{2}F_{1}\left( \alpha + \beta + 1, \alpha, \alpha + \beta, a \right) },
    \end{gathered}
\end{equation}
\begin{equation}
    \label{eq:201803043529}
    \begin{gathered}
        m_{2}
        =
        \frac
        {  B\left( \alpha+2, \beta \right)  {}_{2}F_{1}\left( \alpha + \beta + 2, \alpha + 2, \alpha + \beta + 2, \alpha \right)  }
        {  B\left( \alpha, \beta \right) {}_{2}F_{1}\left( \alpha + \beta + 1, \alpha, \alpha + \beta, a \right) },
    \end{gathered}
\end{equation}
where ${}_{2}F_{1}\left( . \right)$ is the Hypergeometric function as defined in Appendix \ref{sec:Hypergeometric}.
Assuming the following average variance
\begin{equation}
    \label{eq:201803043542}
    \begin{gathered}
        \frac{ \widehat{\sigma}^{2} }{ 2 }
        =
        \frac
        {  B\left( \alpha+1, \beta+1 \right)  {}_{2}F_{1}\left( \alpha + \beta, \alpha + 1, \alpha + \beta + 2, a \right)  }
        {  B\left( \alpha, \beta \right) {}_{2}F_{1}\left( \alpha + \beta + 1, \alpha, \alpha + \beta, a \right) },
    \end{gathered}
\end{equation}
the optimal variance and drift functions are obtained as below,
\begin{equation}
    \label{eq:201803043494}
    \begin{gathered}
        \frac{ \sigma^{2}(x) }{ 2 }
        =
        x ( 1 - x ) ( 1 - ax ),
    \end{gathered}
\end{equation}
\begin{equation}
    \label{eq:201803043505}
    \begin{gathered}
        \mu(x)
        =
        -\left( \beta \left( 1 - a \right) \right) x + \alpha.
    \end{gathered}
\end{equation}
Using the above results, the optimal convergence rate can be obtained from (\ref{eq:201702181190}).

\subsection{Diffusion Processes with Non-Polynomial Variance Function}

As an example of the diffusion processes where their variance function is a non-polynomial function,
in this subsection, we have provided the optimal convergence rate for the diffusion process with the hyperexponential distribution as the stationary distribution.
The hyperexponential distribution is the description of exponential processes in parallel which can be written as below,
\begin{equation}
    \label{eq:Eq1079}
    \begin{gathered}
        \pi(x)  =   \left( p_1 \eta_1 e^{- \eta_1 x} + p_2 \eta_2 e^{- \eta_2 x } \right),
     \end{gathered}
\end{equation}
where, $p_i, \eta_i \geq 0$ for $i=1,2$ and $p_1+p_2=1$.
The hyperexponential distribution is the convex combination of two gamma functions.
First moment and variance of this distribution can be written as below,
\begin{equation}
    \label{eq:Eq1101}
    \begin{gathered}
        m_1  =   \frac{ p_1 }{ \eta_1 }  +  \frac{ p_2 }{ \eta_2 },
     \end{gathered}
\end{equation}
\begin{equation}
    \begin{gathered}
        m_2  -  m_1^2
        =
        \frac{p_1}{\eta_1^2}  +  \frac{p_2}{\eta_2^2}
        +
        p_1 p_2 \left( \frac{1}{\eta_1} - \frac{1}{\eta_2} \right)^2.
    \end{gathered}
\end{equation}
The
cumulative 
distribution function of the hyperexponential distribution is as following,
\begin{equation}
    \begin{gathered}
        \Pi(x) =  1  -  p_1 e^{- \eta_1 x}  -  p_2 e^{- \eta_2 x }.
    \end{gathered}
\end{equation}
Using (\ref{eq:201803021411}), the variance function can be written as below,
\begin{equation}
    \begin{aligned}
       \frac{\sigma^2(x)}{2}
       =
       \frac{ \widehat{\sigma}^2 }{ 2\left( m_2 - m_1^2 \right) \pi(x) } \times
       \left( p_1 e^{- \eta_1 x} \left( x + p_2 \left( \frac{1}{\eta_1} - \frac{1}{\eta_2} \right) \right) - p_2 e^{- \eta_2 x } \left( x + p_1 \left( \frac{1}{\eta_2} - \frac{1}{\eta_1} \right) \right) \right),
     \end{aligned}
\end{equation}
and based on (\ref{eq:201702181190}), for the optimal convergence rate $(\lambda_1)$, we have
\begin{equation}
    \nonumber
    \begin{gathered}
        \lambda_1
        =
         \frac
         {  \widehat{\sigma}^{2} / 2  }
         {  \frac{p_1}{\eta_1^2}  +  \frac{p_2}{\eta_2^2}  +  p_1 p_2 \left( \frac{1}{\eta_1} - \frac{1}{\eta_2} \right)^2  }
         .
    \end{gathered}
\end{equation}
\section{Conclusions}
\label{sec:Conclusions}

Considering all diffusion processes with the given stationary distribution and average variance, here in this paper, we seek the diffusion process with the optimal convergence rate.
In doing so, we have shown that
the optimal drift function is linear and
the ratio of the average variance to the variance of the stationary distribution is an upper bound on the convergence rate of the diffusion processes with given average variance and stationary distribution.
Furthermore, we have proved that
the optimal relaxation time is a concave function of the stationary distribution.
An interesting topic worthy of further investigation is the extension of these analysis to fractional diffusion processes.

\appendix

\section{Optimal Results (Solution)}
\label{Sec:DetailedSolution}

Integrating (\ref{eq:Eq609}),
$\phi_1(x)$
can be written as a first degree polynomial of $x$, i.e.
\begin{equation}
    \label{eq:201803034210}
    \begin{gathered}
        \phi_1(x)  =  \sqrt{\nu} x + c,
    \end{gathered}
\end{equation}
where $c$ is a constant.
Hence, based on (\ref{eq:Eq634}),
$\mu(x)$ %
can be written as
$\mu(x)  =  -\lambda_{1} x - \left( \lambda_{1} / \sqrt{\nu} \right) c$. %
For given stationary distribution $\left( \pi(x) \right)$, domain of $x$ (i.e., $x_1$, $x_2$) and the average variance $\widehat{\sigma}^{2}/2$,
from
$\int_{x_1}^{x_2} { \phi_1(x) \pi(x) dx }  =  0$,
while considering
(\ref{eq:201803034210}),
we have
\begin{equation}
    \label{eq:Eq703}
    \begin{gathered}
        \sqrt{\nu} m_1  + c  =  0,
     \end{gathered}
\end{equation}
where $m_1$ is the first momentum of $\pi(x)$, (i.e., $m_1  =  \int_{x_1}^{x_2} {x \pi(x) dx}$).
On the other hand, considering
(\ref{eq:201803034210}),
from $\int_{x_1}^{x_2} { \pi(x) \left( \phi_1(x) \right)^2 dx }  =  1$,
we have
\begin{equation}
    \label{eq:Eq722}
    \begin{gathered}
        \nu m_2 + 2c\sqrt{\nu}m_1 + c^2  =  1,
     \end{gathered}
\end{equation}
where $m_2$ is the second momentum of $\pi(x)$, (i.e., $m_2  =  \int_{x_1}^{x_2} {x^2 \pi(x) dx}$).
From (\ref{eq:Eq703}) and (\ref{eq:Eq722}), for $\nu$ and $c$, we have
\begin{equation}
    \label{eq:Eq741}
    \begin{gathered}
        \nu = 1 / \left( m_2 - m_1^2 \right),
     \end{gathered}
\end{equation}
\begin{equation}
    \label{eq:Eq749}
    \begin{gathered}
        c = \frac{ - m_1 }{ \sqrt{m_2 - m_1^2} }.
    \end{gathered}
\end{equation}
Note that here $\nu$ is the inverse of standard deviation of $\pi(x)$ and $c$ is the inverse negative of the coefficient of variation of $\pi(x)$.
Substituting (\ref{eq:Eq741}) and (\ref{eq:Eq749}) in (\ref{eq:201803034210}), $\phi_1(x)$ can be written %
as in (\ref{eq:Eq760}).
Thus,
from %
(\ref{eq:Eq760}) and (\ref{eq:Eq634}),
$\mu(x)$  %
is obtained as in (\ref{eq:201803031404}),
and
$\frac{1}{2}\sigma^{2}(x)$  %
(given in (\ref{eq:201803021411}))    %
can be obtained from integrating
(\ref{eq:20170727442}).
Regarding the optimal convergence rate $\left( \lambda_1 \right)$ (given in (\ref{eq:201702181190})), $\lambda_1$ can be obtained by substituting (\ref{eq:Eq741}) in (\ref{eq:201802171163}).

The constraint
(\ref{eq:Eq8}) %
should be satisfied, for all $Q(x)$.
From (\ref{eq:Eq760}), it is obvious that $\frac{\partial}{\partial x}Q(x) = \sqrt{\nu}$.
Substituting this in constraint
(\ref{eq:Eq8}), %
the following  can be concluded,
\begin{equation}
    \label{eq:Eq829}
    \begin{gathered}
        \left. \frac{1}{2}\sigma^{2}(x) \pi(x) \right|_{x = x_1, x_2}  =  0.
     \end{gathered}
\end{equation}

From
(\ref{eq:201803021411}),    %
we have
\begin{equation}
    \label{eq:Eq843}
    \begin{gathered}
        \left. \frac{1}{2}\sigma^{2}(x) \pi(x) \right|_{x = x_1, x_2}  =  0,
        \;\;\;\;
        \left. \int_{x_1}^{x} \mu(z) \pi(z) dz \right|_{x = x_1, x_2}  =  0.
     \end{gathered}
\end{equation}
It is obvious that constraint (\ref{eq:Eq843}) holds true for $x=x_1$.
Regarding the case of $x=x_2$, we have
\begin{equation}
    \label{eq:Eq857}
    \begin{gathered}
        \int_{x_1}^{x_2} \mu(z) \pi(z) dz
        =
        \int_{x_1}^{x_2} \frac{-\lambda_{1}}{\sqrt{\nu}} \phi_1(z) \pi(z) dz
        =
        0.
     \end{gathered}
\end{equation}
The equalities above are based on (\ref{eq:Eq634}) and (\ref{eq:Eq386}) (for $m=1$, $n=0$), respectively.

\section{Hypergeometric Function}
\label{sec:Hypergeometric}
Hypergeometric function \cite{LegendreBook1968,LegendreBookHilbert1965}  %
is defined as below,
\begin{equation}
    \nonumber
    \begin{gathered}
 {}_2F_1(\alpha, \beta, \gamma, \xi)= \sum_{r=0}^{\infty}\frac{(\alpha)_r(\beta)_r\xi^r}{(\gamma)_rr!},
    \end{gathered}
\end{equation}
where  $(\alpha)_r$ is defined  as $(\alpha)_r = \alpha(\alpha + 1)\cdots (\alpha + r - 1)$.

\bibliography{Diffusion_Process_ArXiv_Format}

\begin{thebibliography}{10}

\bibitem{Albanese2001D02}
C.~Albanese, G.~Campolieti, P.~Carr, and A.~Lipton.
\newblock {Black-Scholes} goes hypergeometric.
\newblock {\em Risk Magazine}, 14:99--103, 2001.

\bibitem{LegendreBook1968}
W.~W. Bell.
\newblock {\em Special Functions for Scientists and Engineers}.
\newblock Windsor House, 46 Victoria Street, London, UK., 1968.

\bibitem{Bishwal2008}
J.P.N. Bishwal.
\newblock {\em Parameter Estimation in Stochastic Differential Equation}.
\newblock Springer, New York, 2008.

\bibitem{Borodin2017}
Andrei~N Borodin.
\newblock {\em Stochastic Processes}.
\newblock Springer-Verlag, New York, NY, 2017.

\bibitem{LegendreBookHilbert1965}
R.~Courant and D.~Hilbert.
\newblock {\em Methods of Mathematical Physics}.
\newblock Ser. Wiley classics library. John Wiley and Sons, 1965.

\bibitem{EkelandBook}
I.~Ekeland and R.~Teman.
\newblock {\em Convex Analysis and Variational Problems}.
\newblock North Holland, Amsterdam, 1976.

\bibitem{EXARCHOS2018159}
Ioannis Exarchos and Evangelos~A. Theodorou.
\newblock Stochastic optimal control via forward and backward stochastic
  differential equations and importance sampling.
\newblock {\em Automatica}, 87:159 -- 165, 2018.

\bibitem{Farsad2016}
N.~Farsad, H.~B. Yilmaz, A.~Eckford, C.~B. Chae, and W.~Guo.
\newblock A comprehensive survey of recent advancements in molecular
  communication.
\newblock {\em IEEE Communications Surveys Tutorials}, 18(3):1887--1919,
  thirdquarter 2016.

\bibitem{Hwang1993}
C.~R. Hwang, S.~Y. Hwang-Ma, and Sheu~S. J.
\newblock Accelerating gaussian diffusions.
\newblock {\em Ann. Appl. Probab.}, 3(3):897--913, 1993.

\bibitem{Iacus2008}
S.~Iacus.
\newblock {\em Simulation and Inference for Stochastic Differential Equations}.
\newblock Springer, New York, 2008.

\bibitem{Krylov1980}
N.~V. Krylov.
\newblock {\em Controlled Diffusion Processes}.
\newblock Springer Verlag, 1980.

\bibitem{Krylov1995}
N.~V. Krylov.
\newblock Introduction to the theory of diffusion processes.
\newblock {\em Translations of Mathematical Monographs 142. AMS Providence,
  Rhode Island}, 1995.

\bibitem{LEONENKO201030}
N.N. Leonenko and N.~Šuvak.
\newblock Statistical inference for reciprocal gamma diffusion process.
\newblock {\em Journal of Statistical Planning and Inference}, 140(1):30 -- 51,
  2010.

\bibitem{Linetsky2006}
V.~Linetsky.
\newblock Pricing equity derivatives subject to bankruptcy.
\newblock {\em Math. Finance}, 16:255--282, 2006.

\bibitem{martinbibby2005}
Bo~Martin~Bibby, Ib~Michael~Skovgaard, and Michael {S{\o}rensen}.
\newblock Diffusion-type models with given marginal distribution and
  autocorrelation function.
\newblock {\em Bernoulli}, 11(2):191--220, 04 2005.

\bibitem{Mendoza2009}
R.~Mendoza, P.~Carr, and V.~Linetsky.
\newblock Time changed {Markov} processes in unified creditequity modeling.
\newblock {\em Math. Finance}, 20:527--569, 2010.

\bibitem{Pearson1914}
K.~Pearson.
\newblock {\em Tables for Statisticians and Biometricians}.
\newblock Cambridge University Press, Cambridge, 1914.

\bibitem{Pereyra2016}
Marcelo Pereyra, Philip Schniter, Emilie Chouzenoux, Jean-Christophe Pesquet,
  Jean-Yves Tourneret, Alfred~O Hero, and Steve McLaughlin.
\newblock A survey of stochastic simulation and optimization methods in signal
  processing.
\newblock {\em IEEE Journal of Selected Topics in Signal Processing},
  10(2):224--241, 2015.

\bibitem{Peskir2006}
Goran Peskir.
\newblock {\em On the Fundamental Solution of the Kolmogorov--Shiryaev
  Equation}, pages 535--546.
\newblock Springer Berlin Heidelberg, Berlin, Heidelberg, 2006.

\bibitem{William2015}
William~T. Shaw and Marcus Schofield.
\newblock A model of returns for the post-credit-crunch reality: hybrid
  {Brownian} motion with price feedback.
\newblock {\em Quantitative Finance}, 15(6):975--998, 2015.

\bibitem{Wu2014}
S.~J. Wu, C.~R. Hwang, and M.~T. Chu.
\newblock Attaining the optimal gaussian diffusion acceleration.
\newblock {\em J. Stat. Phys.}, 155:571--590, 2014.

\end{thebibliography}

\end{document}